\documentclass{article}
\usepackage{tikz, hyperref}
\usepackage{amsmath,amscd}
\usepackage{amssymb}
\usepackage{upref}
\usepackage{multirow}
\usepackage{multicol}
\usepackage{url}
\usepackage[all]{xy}

\usepackage{color}
\usepackage{lineno}
\usepackage{makeidx}
\makeindex
\usepackage{theorem}
\newtheorem{theorem}{Theorem}

\newtheorem{proposition}[theorem]{Proposition}
\newtheorem{corollary}[theorem]{Corollary}
{\theorembodyfont{\rmfamily}%
  \newtheorem{example}[theorem]{Example}
   }
\newenvironment{proof}{\noindent\textit{Proof.}}
{\QED\vskip\theorempostskipamount} 

\def\petitcarre{\vrule height4pt width 4pt depth0pt}
\def\QED{\relax\ifmmode\eqno{\hbox{\petitcarre}}\else{%
  \unskip\nobreak\hfil\penalty50\hskip2em\hbox{}\nobreak\hfil
  \petitcarre
  \parfillskip=0pt \finalhyphendemerits=0\par\smallskip}
  \fi}
\newcommand\cL{\mathcal{L}}

\newcommand{\cP}{\mathcal{P}}

\newcommand{\N}{\mathbb{N}}
\newcommand{\Z}{\mathbb{Z}}

\newcommand{\CC}{\mathbb{C}}

\def\un(#1){\underline{#1}\,}

\DeclareMathOperator{\Card}{Card}

\definecolor{ivoire}{rgb}{0.99,0.99,0.8}

\definecolor{light-gray}{gray}{0.7}
\usepackage{xspace}
\newcommand{\resp}{{resp.}\xspace}
\usepackage{calc}
\newcounter{hours}\newcounter{minutes}

\numberwithin{theorem}{section}
\numberwithin{equation}{section}
\numberwithin{figure}{section}
\numberwithin{table}{section}

\definecolor{lime}{HTML}{A6CE39}
\DeclareRobustCommand{\orcidicon}{%
	\begin{tikzpicture}
	\draw[lime, fill=lime] (0,0)
	circle [radius=0.16]
	node[white] {{\fontfamily{qag}\selectfont \tiny ID}};
	\draw[white, fill=white] (-0.0625,0.095)
	circle [radius=0.007];
	\end{tikzpicture}
	\hspace{-2mm}
}
\foreach \x in {A, ..., Z}{%
 \expandafter\xdef\csname orcid\x\endcsname{\noexpand%
 \href{https://orcid.org/\csname orcidauthor\x\endcsname}{\noexpand\orcidicon}}
}

\title{A note on one-sided recognizable morphisms}

\author{Marie-Pierre B\'eal\orcidA{},
  Val\'erie Berth\'e$^1$,
  Dominique Perrin$^2$\\ and Antonio Restivo$^3$\\
  $^1$ IRIF, 
$^2$ LIGM, Universit\'e Gustave Eiffel, $^3$ Universita di Palermo}
\begin{document}

\maketitle
\begin{abstract}
  We revisit the notion of one-sided recognizability of morphisms
  and its relation to two-sided recognizability.
\end{abstract}
\tableofcontents
\section{Introduction}

The notion of recognizability for morphisms is an important
one with a long history (see~\cite{Kyriakoglou2019} for an account of it).

The first attempts used a one-sided notion adapted to one-sided
infinite sequences. The main progress realized with Moss\'e's Theorem
(see Theorem~\ref{theoremMosse} below) was made possible by turning to a two-sided version
of recognizability. Since then, several generalizations of
Moss\'e's theorem have been obtained
(see \cite{BezuglyiKwiatkowskiMedynets2009}, \cite{BertheSteinerThuswaldnerYassawi2019} and \cite{BealPerrinRestivo2021}).

In this note, we come back to the one-sided version of recognizability.
It was studied in
\cite{CrabbDuncanMcGregor2010} for alphabets with
two letters and more recently in~\cite{AkiyamaTanYuasa2017}. The authors of \cite{CrabbDuncanMcGregor2010}
prove that although a primitive aperiodic endomorphism on two letters is not always
one-sided recognizable, it is almost so, in the sense
that the one-sided sequences for
which this fails have a special form. We generalize their result to arbitrary
primitive aperiodic endomorphisms.

More precisely, given a shift space $X$ on the alphabet $A$,
we define the recognizability of a morphism
$\sigma\colon A^*\to B^*$ on $X$ by a uniqueness desubstitution
property and we relate this definition with
the original definition of Moss\'e (Proposition~\ref{propositionTwoDefsRec}).
We state without proof the theorem of Moss\'e asserting that
a primitive aperiodic morphism $\sigma$ is recognizable
on the shift $X(\sigma)$ defined by $\sigma$ (Theorem~\ref{theoremMosse}).

We next define the one-sided
recognizablity of a morphism $\sigma$ on a one-sided shift.
We first relate   the notion of (two-sided) recognizability
on a shift $X$ with the one-sided recognizability
on the one-sided shift $X^+$ associated to $X$
(Proposition~\ref{propositionRightMarked}). Next, we
relate it with the original definition
of Moss\'e (Proposition~\ref{proposition12Mosse}).
The main results are
\begin{enumerate}
\item Theorem~\ref{theoremCharactOneSided} which characterizes the
endomorphisms $\sigma$ which are not one-sided recognizable on $X(\sigma)^+$.
\item Theorem~\ref{theoremAlmostRec} which states that every endomorphism $\sigma$
is almost one-sided recognizable on the shift $X(\sigma)$, in the sense that it
is one-sided recognizable except at a finite number of points.
\end{enumerate}

We end the paper with a mention of the point which has motivated us
for this note and concerns continuous eigenvalues of shift spaces.
Indeed, a result of Host~\cite{Host1986} on the eigenvalues
of substitution shifts is formulated using one-sided shifts and  one-sided
recognizability (and was
recently extended to $\cal S$-adic systems in~\cite{BertheCecchiYassawi2022}).
Its proof can be however be read without change using two-sided
shifts (and two-sided recognizability). We  contribute to
the clarification of the situation, giving a simple proof that 
 a recurrent shift space and  its associated one-sided
 shift have the same spectrum (Proposition~\ref{propositionSpectra12},
 see also \cite[Proposition 2.1]{BertheCecchiYassawi2022}).

%%%%%%%%%%%%%%
\section{Shift spaces}
Let $A$ be a finite alphabet. We denote by $A^*$ the set of words on $A$,
by $\varepsilon$ the empty word, and by $A^+$ the set of nonempty words.

We consider the set $A^\Z$ of two-sided sequences of elements of $A$
and the corresponding set $A^\N$ of one-sided sequences.
For $x\in A^\Z$ and $i\le j$, we denote $x_{[i,j]}=x_{i}x_{i+1}\cdots x_j$
and $x_{[i,j)}=x_{i}x_{i+1}\cdots x_{j-1}$.

For a nonempty word $w\in A^+$, we denote $w^\omega=www\cdots$
and $w^\infty=\cdots ww\cdot www\cdots$ (where the index $0$ is at
the beginning of $w$).

We denote by $S$ the \emph{shift transformation} defined for
$x\in A^\Z$ (\resp $x\in A^\N$)
by $S(x)=y$ if $y_n=x_{n+1}$ for $n\in\Z$ (\resp $n\in \N$).
A subset $X$ of $A^\Z$ (\resp $A^\N$) is \emph{shift invariant} if
$S(X)=X$.

The \emph{orbit} of a point $x\in A^\Z$
is the set of all $S^n(x)$ for $n\in \Z$.

The set $A^\Z$ of two-sided infinite sequences of elements of
$A$ is a compact metric space for the distance $d(x,y)=2^{-r(x,y)}$
with
\begin{displaymath}
  r(x,y)=\max\{n\ge 0\mid x_{[-n,n]}=y_{[-n,n]}\}.
  \end{displaymath}
Similarly, the set $A^\N$ of one-sided sequences of elements of $A$
is a compact metric space for the distance defined using
$r(x,y)=\max\{n\ge 0\mid x_{[0,n]}=y_{[0,n]}\}$.

A \emph{shift space} (\resp one-sided shift space)
on the alphabet $A$
is a closed and shift invariant subset of $A^\Z$ (\resp $A^\N$).
The set $A^\Z$ (\resp $A^\N$) itself is a a shift space
(\resp a one-sided shift space) called the \emph{full shift}
(\resp the full one-sided shift).

The shift space \emph{generated} by a sequence $x\in A^\Z$
(\resp $x\in A^\N$) is the topological closure of the
set $\cup_{n\in \Z}S^nx$ (\resp $\cup_{n\in \N}S^nx$). It is
the smallest shift space (\resp one-sided shift space)
containing $x$.

A shift space is a particular case of a (topological) \emph{dynamical system},
which is by definition a pair $(X,T)$ of a compact metric space $X$
and a continuous map $T$ from $X$ into itself.
It is \emph{invertible} if $T$ is invertible (and thus a homeomorphism).

A morphism from a dynamical system $(X,T)$ to a dynamical system
$(X',T')$ is a continuous map $\varphi\colon X\to X'$ which interleaves
with $T,T'$, that is, such that $T'\circ\varphi=\varphi\circ T$.

Given an invertible system $(X,T)$, the \emph{orbit} of $x\in X$ is
the set $\{T^n(x)\mid n\in \Z\}$. Its \emph{forward orbit}
is the set $\{T^n(x)\mid n\ge 0\}$.

The \emph{language} of a shift space
(\resp a one-sided shift space)  $X$, denoted $\cL(X)$
is the set of factors of the sequences in $X$.
We denote by $\cL_n(X)$ the set of words of length $n$
in $\cL(X)$.

Let $X,Y$ be shift spaces on alphabets $A,B$ respectively.
Given an integer $N$,
a \emph{block map} of \emph{window size} $N$
is a map $f\colon \cL_{2N+1}(X)\to B$. The \emph{sliding block code}
defined by $f$ is the map $\varphi:X\to B^\Z$ defined by
$\varphi(x)=y$ if
\begin{displaymath}
  y_n=f(x_{[n-N,n+N]})\quad (n\in\Z)
\end{displaymath}
By a classical result, a map $\varphi\colon X\to Y$ is a morphism
if and only if it is a sliding block code from $X$ into $Y$
\cite[Theorem 6.2.9]{LindMarcus1995}.

For a two-sided sequence $x\in A^\Z$, we denote $x^+=x_0x_1\cdots$.
If $X$ is a shift space, we denote by $X^+$ the set of
$x^+$ for $x\in X$. It is a one-sided shift space.
Note that $X$ is determined
by $X^+$ since for every shift space $X$, one has the equality
\begin{displaymath}
  X=\{x\in A^\Z\mid x_nx_{n+1}\cdots\in X^+,\mbox{for all $n\in\Z$}\}.
\end{displaymath}
Thus, the map $X\mapsto X^+$ is a bijection from the family
of shift spaces on $A$ onto the family of one-sided shift
spaces on $A$.

A sequence $x\in A^\Z$ (\resp $x\in A^\N$) is \emph{periodic}
if $S^n(x)=x$ for some $n\ge 1$. Otherwise, it is
\emph{aperiodic}. A shift space (\resp a one-sided shift space) is
periodic if all its elements are periodic. It is
aperiodic if all its elements  are aperiodic.

A topological dynamical system is \emph{recurrent} if
there is a point  with a dense forward orbit.

A nonempty topological dynamical system is \emph{minimal}
if, for every closed subset $Y$ of $X$ such
that $T(Y)\subset Y$, one has $Y=\emptyset$ or $Y=X$.
Equivalently, $X$ is minimal if and only if
the orbit  of every point $x$
is dense.

A shift space $X$ is \emph{recurrent}
if for every $u,v\in \cL(X)$ there is a word $w$ such that $uwv\in\cL(X)$.

A shift space $X$ is
\emph{uniformly recurrent}
if for every $w\in\cL(X)$ there is an $n\ge 1$
such that $w$ is a factor of every word in $\cL_n(X)$.

Given a word $u\in\cL(X)$, a \emph{right returm word}
 to $u$
is a nonempty word $w$ such that $uw\in\cL(X)$  and
that $uw$ has exactly two occurrences of $u$,
one as a prefix and one as a suffix.

Similarly, a \emph{left return word} to $u$ is a nonempty
word $w$ such that $wu\in\cL(X)$  and
that $wu$ has exactly two occurrences of $u$,
one as a prefix and one as a suffix.

A shift space is uniformly recurrent if and only if
it is recurrent and
for every $u\in \cL(X)$ the set of return words to $u$ is finite.

The following is well known (see~\cite{DurandPerrin2021} for example).

\begin{proposition}
A shift space is minimal if and only if it is uniformly recurrent.
\end{proposition}

Let $X$ be a shift space on $A$. For $w\in \cL(X)$, we denote
\begin{displaymath}
  \ell(w)=\Card\{a\in A\mid aw\in\cL(X)\},\quad r(w)=\Card\{a\in A\mid wa\in\cL(X)\}.
  \end{displaymath}
A word $w\in\cL(X)$ is \emph{left-special} (\resp \emph{right-special})
if $\ell(w)\ge 2$
(\resp $r(w)\ge 2$).

Let $X$ be a one-sided shift. For a one-sided sequence $x\in X$,
we denote $\ell(x)=\Card\{a\in A\mid ax\in X\}$.
A one-sided sequence $x\in A^\N$ 
is \emph{left-special} if $\ell(x)\ge 2$.

%\begin{example}
%  Let $\sigma\colon a\mapsto ab,b\mapsto a$ be the Fibonacci morphism.
%  There is one left-special sequence in $X(\sigma)^+$, namely
%  $x=\sigma^\omega(a)$. 
%\end{example}

Two points $x,y$ of a two-sided
shift space $X$ are \emph{right asymptotic} if there is an $n\ge 0$
such that $S^n(x)^+=S^n(y)^+$. They are \emph{asymptotically equivalent}
if there are $n,m\in \Z$ such that $S^n(x)^+=S^m(y)^+$.
The classes of this equivalence are called the \emph{asymptotic classes}.
Every class is a union of orbits.
An asymptotic class is non-trivial if it is not reduced to one orbit.

The \emph{complexity} of a shift space $X$ is the sequence
$p_n(X)=\Card(\cL_n(X))$. We denote $s_n(X)=p_{n+1}(X)-p_n(X)$.
It is classical that
\begin{equation}
  s_n(X)=\sum_{w\in \cL_n(X)}(\ell(w)-1)=\sum_{w\in\cL_n(X)}(r(w)-1).\label{eqsn}
\end{equation}
Indeed, one has
\begin{eqnarray*}
  s_n(X)&=&p_{n+1}(X)-p_n(X)=\Card(\cL_{n+1}(X))-\Card(\cL_{n}(X))\\
  &=&\sum_{w\in\cL_{n}(X)}(\ell(w)-1).
\end{eqnarray*}
If the sequence $s_n(X)$ is bounded, the complexity $p_n(X)$
is at most linear, that is $p_n(X)\le kn$ for some $k\ge 1$.
The converse is true by an important result due to Cassaigne
\cite{Cassaigne1996}.
\begin{proposition}\label{propositionCassaigne}
  If the complexity of a shift $X$ is at most linear,
  then $s_n(X)$ is bounded.
  \end{proposition}
%Two sequences $x,y$ in a shift space $X$ are \emph{right asymptotic}
%if there is an $n\ge 0$ such that $S^n(x)=S^n(y)$. An
%\emph{asymptotic pair} is a pair $(x,y)$ of asymptotic sequences

%Let $X$ be a shift space. A \emph{return word} to $w\in\cL(X)$
%is a nonempty word $u$ such that $wu\in \cL(X)$ and $wu$ has exactly
%two occurrences of $w$, one as a prefix and the other one as a suffix.

A shift space $X$ is \emph{linearly recurrent} if there is a constant
$K$ such that for every $w\in\cL(X)$, the length of every
return word to $w$ is bounded by $K|w|$.

The following result is from \cite{Durand&Host&Skau:1999}.
\begin{proposition}\label{propositionLRisLinear}
Every linearly recurrent shift has at most linear complexity.
\end{proposition}
Note that this implies that,
in a linearly recurrent shift, the sequence $s_n(X)$
is bounded and thus, by \eqref{eqsn}, the number of left-special
words of length $n$ is bounded, and finally that the number of
left-special
one-sided infinite sequences is finite. In this case, every one-sided
sequence in $X^+$ but a countable number of them, has a unique left
extension in $X$. In particular, the shifts $X$ and $X^+$
are measurably isomorphic (see~\cite{DurandPerrin2021}).
%%%%%%%%%%%%%%%%%%
\section{Morphisms and recognizability}
Let $\sigma\colon A^*\to B^*$ be a morphism.
Then $\sigma$ extends to a map from $A^\Z$ (\resp $A^\N)$
to $A^*\cup A^\Z$ (\resp $A^*\cup A^\N$).

Let $\sigma\colon A^*\to B^*$ be a morphism.
A letter $a\in A$ is \emph{erasable} if $\sigma^n(a)=\varepsilon$
for some $n\ge 1$. The morphism is \emph{non-erasing}
if there is no erasable letter.

Let $\sigma\colon A^*\to B^*$ be a morphism.
A \emph{$\sigma$-representation} of a point $y\in B^\Z$
 is a pair
$(x,k)$ with $x\in A^\Z$ 
and $0\le k<|\sigma(x_0)|$ such that
$y=S^k(\sigma(x))$.

Let $X$ be a shift space on $A$.
A morphism $\sigma\colon A^*\to B^*$ is \emph{recognizable}
on $X$  at $y\in B^\Z$
if $y$ has at most one $\sigma$-representation $(x,k)$
with $x\in X$.
It is recognizable on $X$
 if it is recognizable
 on $X$ at every $y\in B^\Z$.

 For a shift space $X$ on $A$ and a word $w$ of length $n$, we denote
 $[w]_X=\{x\in X\mid x_{[0,n)}=w\}$.

 \begin{proposition}\label{propositionPartition}
   Let  $\sigma\colon A^*\to B^*$ be a morphism, let $X$
   be a shift space on $A$ and let $Y$ be the closure under the
   shift of $\sigma(X)$. The morphism $\sigma$
   is recognizable on $X$ if and only if  the family $\cP$
of sets $S^k\sigma([a]_X)$ for $a\in\cL_1(X)$ and $0\le k<|\sigma(a)|$
   forms a partition of $Y$.
 \end{proposition}
 \begin{proof}
   Assume first that $\sigma$ is recognizable on $X$. Every element
   of $Y$ has a $\sigma$-representation and thus the union of
   the elements of $\cP$ is $Y$. Next if two of them intersect,
   then some element of $Y$ has two distinct $\sigma$-representations,
   which is impossible.

   Conversely, if $\sigma$ is not recognizable on $X$, there is
   some $y$ with two distinct $\sigma$ representations
   $(x,k)$ and $(x',k')$.
   If $x=x'$ then $S^k\sigma([x_0])$ and $S^{k'}\sigma([x_0])$
   are two distinct elements of the family $\cP$ with nonempty intersection.
   Otherwise, shifting $y$ if necessary, we may assume that
   $x_0\ne x'_0$, whence the conclusion that $\cP$ is not a partition
   again.
 \end{proof}
 The partition $\cP$ above, called a \emph{partition in towers},
 plays an important role in the
 definition of a Bratteli diagram associated to a substitution shift
 (see~\cite{DurandPerrin2021}).
 \begin{example}
   The morphism  $\sigma\colon a\mapsto ab, b\mapsto a$
   is called the \emph{Fibonacci morphism}.
   It is recognizable on $X=A^\Z$
   since the family
   \begin{displaymath}
[ab]_Y, [b]_Y, [aa]_Y
   \end{displaymath}
   forms a partition of $Y$. One has
   $\sigma([b]_X)=[aa]_Y$ because $\sigma(a)$ and $\sigma(b)$ begin with $a$.
 \end{example}

 %Proposition~\ref{propositionPartition} is false for an erasing morphism,
 %as shown by the following example.
 %\begin{example}
 %  Let $\sigma\colon a\mapsto a, b\mapsto\varepsilon$. The morphism
 %  $\sigma$ is not recognizable on $A^\Z$ since for example
 %  $\sigma(\cdots aa\cdot aa)=\sigma(\cdots aa\cdot baa\cdots)$.
 %  However, the family $\cP$ is reduced to $[a]$.
 %  \end{example}

 The definition of recognizability given above is a dynamical one
 and the
one in current use now (see~\cite{BertheSteinerThuswaldnerYassawi2019}
for example)
but it was given in a different
form (and only for endomorphisms)
in the articles of Moss\'e \cite{Mosse1992,Mosse1996}.

Let $\sigma\colon A^*\to B^*$ be a morphism, let  $x\in A^\Z$
be such that $y=\sigma(x)$ is two-sided infinite. Define the set of \emph{cutting points} of $x$ as
\begin{displaymath}
  C(x)=\{|\sigma(x_{[0,n)})|\mid n\ge 0\}\cup \{-|\sigma(x_{[n,0)})|\mid n<0\}.
\end{displaymath}
Given $N\ge 1$, let us say that
$\sigma$ is  \emph{recognizable in the sense of Moss\'e} for $x$
with scope $N$ if for $i,j\in\Z$, whenever
$y_{[i-N,i+N]}=y_{[j-N,j+N]}$, then $i\in C(x)\Leftrightarrow j\in C(x)$.

The following result connecting the two notions of
recognizability is proved in~\cite[Theorem 2.5]{BertheSteinerThuswaldnerYassawi2019}. All morphisms are supposed to be non-erasing
in~\cite{BertheSteinerThuswaldnerYassawi2019}, but the proof of (i)
remains the same in the general case.
\begin{proposition}\label{propositionTwoDefsRec}
  Let $\sigma\colon A^*\to B^*$ be  morphism,
  let $x\in A^\Z$ be such that $\sigma(x)$
  is two-sided infinite and let $X$ be the shift generated by $x$.
  The following assertions hold.
  \begin{enumerate}
  \item[\rm (i)] If $\sigma$ is  recognizable on $X$
    then it is recognizable in the sense of Moss\'e for $x$.
    \item[\rm(ii)] If the shift $X$ is minimal, if the morphism $\sigma$
      is non-erasing, injective on $A$ and  recognizable in the sense of
      Moss\'e for $x$, then $\sigma$ is recognizable on $X$.
      \end{enumerate}
\end{proposition}
Assertion (ii) is not true without its restrictive hypotheses on $X$ and
$\sigma$. Indeed, for example, if $\sigma\colon a\mapsto a,b\mapsto\varepsilon$,
then $\sigma$ is recognizable in the sense of Moss\'e for every $x\in A^\Z$
with a finite number of $b$, but it is not recognizable on the
shift $X$ generated by $x$ since $\sigma(X)=a^\infty$.

The notion of recognizability is closely related to the notion
of \emph{tower construction} that we recall now. Given a morphism
$\sigma\colon A^*\to B^*$ and a shift space $X$ on $A$, let $(X^\sigma,T)$ be the dynamical system
defined by
\begin{equation}
  X^\sigma=\{(x,k)\mid x\in X, 0\le k<|\sigma(x_0)|\}\label{eqXsigma}
\end{equation}
and
\begin{displaymath}
  T(x,k)=\begin{cases}(x,k+1)&\mbox{ if $k+1<|\sigma(x_0)|$}\\
  (S(x),0)&\mbox{ otherwise}
  \end{cases}
\end{displaymath}
Then the map $\hat{\sigma}\colon (x,i)\mapsto S^i\sigma(x)$
is a morphism of dynamical systems from $(X^\sigma,T)$
onto the shift $Y$ which is the closure under the shift of $\sigma(X)$.
The morphism $\sigma$ is recognizable on $X$ if and only if
$\hat{\sigma}$ is a homeomorphism.

Note that we may consider $(X^\sigma,T)$ as a shift space
on the alphabet
\begin{equation}
  A^\sigma=\{(a,k)\mid a\in A, 0\le k<|\sigma(a)|\}.\label{eqAsigma}
  \end{equation}
Indeed, there is a unique morphism $\alpha$ from $(X^\sigma,T)$
into $(A^\sigma)^\Z$
 such that
\begin{equation}
  \alpha(x,k)_0=(x_0,k)\label{eqalpha}
\end{equation}

%Note also that another equivalent definition of recognizability
%uses partitions. Indeed, the morphism $\sigma\colon A^*\to B^*$
%is recognizable if and only if the family
%\begin{displaymath}
%  \Pg=\{S^i\sigma([a])\mid a\in A,0\le i<|\sigma(a)|\}
%\end{displaymath}
%is a partition. Such a partition is called a partition in towers
%(see~\cite{DurandPerrin2021}).
 
\paragraph{Endomorphisms}
A morphism $\sigma\colon A^*\to A^*$ is called an
\emph{endomorphism}.
Let $\sigma\colon A^*\to A^*$ be an endomorphism. The
language $\cL(\sigma)$ is the set of factors of
the words $\sigma^n(a)$ for some $n\ge 0$ and some $a\in A$.
The shift $X(\sigma)$ is the set of $x\in A^\Z$ with all
their factors in $\cL(\sigma)$. Such a shift is
called a \emph{substitution shift}.

Endomorphisms are often called \emph{substitutions}
(in general with additional requirements, such as being
non erasing and with $\cL(\sigma)=\cL(X(\sigma))$ as in~\cite{DurandPerrin2021}).

The troubles arising with erasable letters are simplified
for substitution shifts  since by \cite[Lemma 3.13]{BealPerrinRestivo2022},
for every $x\in X(\sigma)$, the sequence $\sigma(x)$ is in $X(\sigma)$
(in particular, $\sigma(x)$ is a two-sided infinite sequence).

An endomorphism $\sigma\colon A^*\to A^*$ is \emph{primitive}
if there is an integer $n\ge 1$ such that for every $a,b\in A$
one has $|\sigma^n(a)|_b\ge 1$.
The following is well known (see \cite{DurandPerrin2021} for example).
\begin{proposition}
  If $\sigma\colon A^*\to A^*$ is primitive
  and $\Card(A)\ge 1$, then $X(\sigma)$ is minimal.
\end{proposition}

The following result is from \cite{Damanik&Lenz:2006} (see also \cite{DurandPerrin2021}).
\begin{proposition}\label{propositionMinimalLR}
Every minimal substitution shift is linearly recurrent.
\end{proposition}
Combining Propositions \ref{propositionMinimalLR} and
\ref{propositionLRisLinear}, we obtain that every minimal
substitution shift has at most linear complexity.

Note that, by Equation~\eqref{eqsn}, this implies that
for a minimal substitution shift  $X(\sigma)$, the number of left-special
sequences is finite.
\begin{example}
  The Fibonacci morphism $\sigma\colon a\mapsto ab,b\mapsto a$
  is primitive. The complexity of $X(\sigma)$ is $p_n(X)=n+1$.
  As a minimal shift of complexity $n+1$, the shift $X(\sigma)$ is,
  by definition, a Sturmian shift.
  The sequence $\sigma^\omega(a)$  is
  the unique one-sided sequence having all $\sigma^n(a)$ as
  prefixes. It is the unique fixed point
  of $\sigma$ in $X(\sigma)^+$
  and also the unique left-special sequence in $X(\sigma)^+$.
  \end{example}
As a consequence, we have the following finiteness result.
\begin{proposition}\label{propositionNbAsymptotic}
  Let $\sigma$ be a morphism. If $X(\sigma)$ is minimal,
  the number of non-trivial
  asymptotic classes of $X(\sigma)$ is finite and bounded by the number
  of left-special sequences.
\end{proposition}
\begin{proof}
  If $X(\sigma)$ is periodic, there is no non-trivial
  asymptotic class. Thus, since $X(\sigma)$ is minimal,
  we may assume that $X(\sigma)$
  (and also $X(\sigma)^+$) is
  aperiodic.
  Let $\alpha$ be the map which assigns to a left-special sequence
  $x\in X(\sigma)^+$ the asymptotic class of the points
  $z\in X(\sigma)$ such that $z^+=x$. Since $x$ is left-special
  and since $x$ cannot be periodic,
  the class $\alpha(x)$ is non-trivial.
  Let indeed $z,z'\in X$ be distinct and such that $z'^+=z^+=x$
  (they exist since $x$ is left-special). If $z$ is a shift
  of $z'$, then $x$ is a proper shift of itself and thus
  it is periodic, a contradiction.
  
  The map $\alpha$ is surjective from the set of left-special
  sequences to the set of non-trivial asymptotic classes.
  Indeed, let $C$ be a non-trivial asymptotic class. Let
  $x,y\in C$ be in distinct orbits and let $n,m\in\Z$ be such that
  $S^nx^+=S^my^+$. Set $z=S^nx$ and $t=S^my$. Then $z,t$ are distinct
  points in $C$ such that $z^+=t^+$ and thus there is at least one
  $u\in C$ such that $u^+$ is left-special.
\end{proof}

Let $X$ be a shift space.
For an asymptotic class $C$ of $X$, we denote $\omega(C) = \Card(o(C))-1$
 where $o(C)$ is the set of orbits contained in $C$.
For a right infinite word $u \in X^+$, let 
\begin{displaymath}
\ell_C(u) = \Card\{ a \in A \mid x^+ = au \mbox{ for some $x\in C$} \}.
\end{displaymath}
We denote by $LS_\omega(C)$ the set of right infinite words $u$ such that $\ell_C(u) \ge 2$.

The following statement is proved in~\cite[Proposition 4.3]{DolcePerrin2020}.

\begin{proposition}\label{pro:eqomegaC}
Let $X$ be a  shift space and let $C$ be a right asymptotic class.
Then
\begin{equation}
\omega(C) = \sum_{u\in LS_\omega(C)}(\ell_C(u)-1)
\label{equationC}
\end{equation}
where both sides are simultaneously finite.
\end{proposition}
The following result is Moss\'e's Theorem (see \cite{DurandPerrin2021} for references). A morphism is \emph{aperiodic} if $X(\sigma)$ is aperiodic.

\begin{theorem}\label{theoremMosse}
  If $\sigma\colon A^*\to A^*$ is a primitive aperiodic morphism,
  then $\sigma$ is recognizable on $X(\sigma)$.
\end{theorem}

The following generalization of Moss\'e's Theorem was proved
in~\cite{BertheSteinerThuswaldnerYassawi2019} for non-erasing morphisms.
A different proof holding in the more general case of morphisms
with erasable letters was given in~\cite{BealPerrinRestivo2021}.
\begin{theorem}\label{theoremRecognizabilityAperiodic}
  Any endomorphism $\sigma\colon A^*\to A^*$ is recognizable on $X(\sigma)$
  at aperiodic points.
\end{theorem}
The following example illustrates the case of an erasing morphism.
\begin{example}
  Let $\sigma\colon a\mapsto ab,b\mapsto ac,c\mapsto \varepsilon$.
  The shift $X(\sigma)$ is the Fibonacci shift with letters $c$
  inserted at the cutting points. The morphism $\sigma$ is
  recognizable on $X(\sigma)$.
  \end{example}
\begin{corollary}\label{corollaryLeftSpecial}
  Let $\sigma\colon A^*\to A^*$ be a morphism such that
  $X(\sigma)$ is minimal and aperiodic.
  Every left-special sequence $x\in X(\sigma)^+$ is
  a fixed point of a power of $\sigma$.
\end{corollary}
\begin{proof}
  Since, by Theorem~\ref{theoremRecognizabilityAperiodic}, $\sigma$ is recognizable
  on $X(\sigma)$ at aperiodic points and since
  $X(\sigma)$ is minimal aperiodic, the morphism
  $\sigma$ is recognizable on $X(\sigma)$. The map $x\mapsto \sigma(x)$ induces
  a bijection from the set of orbits in $X(\sigma)$
  onto itself (as we have seen before, $\sigma(x)$ is
  a two-sided infinite sequence for every $x\in X(\sigma)$). This bijection maps every
  non-trivial asymptotic class onto a non-trivial
  asymptotic class. Since there is a finite number
  of these classes by Proposition~\ref{propositionNbAsymptotic},
  and since each class is formed of a finite number of
  orbits by Proposition~\ref{pro:eqomegaC},
  some power of $\sigma$ fixes each of the orbits forming
  each of these classes.
  \end{proof}
\begin{example}
  Let $\sigma\colon a\mapsto ab, b\mapsto a$ be the Fibonacci
  morphism. The sequence $\sigma^\omega(a)$  is
  the unique one-sided sequence having all $\sigma^n(a)$ as
  prefixes. It is the unique fixed point
  of $\sigma$ and also the unique left-special sequence.
\end{example}
We say that a morphism $\sigma\colon A^*\to B^*$ is
\emph{almost recognizable} 
on a shift $X$ if $\sigma$ is recognizable on $X$ except
at a finite number of points of $B^\Z$.

\begin{theorem}
  Every morphism $\sigma\colon A^*\to A^*$ is almost recognizable
  on $X(\sigma)$.
\end{theorem}
The proof results directly from Theorem~\ref{theoremRecognizabilityAperiodic}
using the following result, proved in~\cite{BealPerrinRestivo2022}.
\begin{theorem}\label{theoremPeriodic}
  Let $\sigma\colon A^*\to A^*$ be a morphism.
The set of periodic points in $X(\sigma)$ is finite.
\end{theorem}
%%%%%%%%%%%%%%%%%%%%%%%%%%%%
\section{One-sided recognizability}
We now focus on one-sided shifts and the corresponding notion
of one-sided recognizability.

Let $\sigma\colon A^*\to B^*$ be a morphism.
As in the case of two-sided sequences,
a $\sigma$-representation of a one-sided sequence  $y\in A^\N$ is a pair
$(x,k)$ with  $x\in A^\N$
and $0\le k<|\sigma(x_0)|$ such that
$y=S^k(\sigma(x))$.

Let $X$ be a one-sided shift space on $A$.
A morphism $\sigma\colon A^*\to B^*$ is \emph{one-sided recognizable}
on $X$ at  $y\in B^\N$
if $y$ has at most one $\sigma$-representation $(x,k)$
with  $x\in X$.
It is  one-sided recognizable if it is recognizable
on $X$ at every  $y\in B^\N$.

As for two-sided recognizability, the notion of
one-sided recognizability can be formulated using
the map $\hat{\sigma}\colon (X^\sigma,T)\to Y$ where
$X^\sigma$ is defined by \eqref{eqXsigma} and $Y$
is the closure of $\sigma(X)$ under the shift.

We first have the following statement describing the relation
between recognizability and one-sided recognizability.
 A non-erasing morphism $\sigma\colon A^*\to B^*$
 is  called
 \emph{right-marked} if the words $\sigma(a)$ for $a\in A$
 end with different letters.
\begin{proposition}\label{propositionRightMarked}
  Let $\sigma\colon A^*\to B^*$ be a morphism and let $X$
  be a shift space on $A$.
  \begin{enumerate}
    \item
If $\sigma$ is one-sided recognizable on $X^+$,
it is recognizable on $X$.
\item If $\sigma$ is right-marked and is recognizable on $X$,
  then $\sigma$ is one-sided recognizable on $X^+$.
  \end{enumerate}
\end{proposition}
\begin{proof}
  1. Let $y\in B^\Z$ have two $\sigma$-representations $(x,k)$ and $(x',k')$
  with $x,x'\in X$. Since $y^+=S^k(\sigma(x^+))=S^{k'}(\sigma(x'^+)$,
  we have $x^+=x'^+$ and $k=k'$. Since we may apply this argument
  for every shift $S^{-n}(y)$, we conclude that $x=x'$.

  2. Let $y\in B^\N$ have two $\sigma$-representations $(z,k)$ and
  $(z',k')$ with $z,z'\in X^+$. Let $t,t'\in X$ be such that
  $t^+=z$ and $t'^+=z'$. Consider the
  system $(X^\sigma,T)$ obtained by the tower construction, with $X^\sigma$
  defined by Equation~\eqref{eqXsigma}. Let $Y$ be the closure of $\sigma(X)$
  under the shift. Since $\sigma$ is recognizable, the
  map $\hat{\sigma}\colon (x,k)\mapsto S^k\sigma(x)$ is a homeomorphism
  from $(X^\sigma,T)$ onto $Y$. Since $X^\sigma$ may be considered as a shift
  space on the alphabet $A^\sigma$, the morphism $\hat{\sigma}^{-1}$
  is a sliding block code defined by a blockmap $f\colon \cL_{2N+1}(Y)\to A^\sigma$
  of window size $N$. Consequently, we have $T^n(t,k)=T^n(t',k')$
  for all $n\ge 2N$. This implies that $z=ux$, $z'=u'x$ and
  $y=vy'$ with $vy'=S^kux=S^{k'}u'x$ and $\sigma(x)=y'$. Since $\sigma$
  is right marked, this implies that $k=k'$ and $u=u'$.
  \end{proof}

\begin{example}\label{examplePeriodDoubling}
  The morphism $\sigma\colon a\mapsto ab,b\mapsto aa$
  is called the \emph{period-doubling morphism}
  and the shift $X(\sigma)$ the \emph{period-doubling shift}
  (see~\cite{DurandPerrin2021}).
  Since it is primitive aperiodic, it is recognizable on $X(\sigma)$.
  Since it is right-marked, it is also one-sided recognizable
  on $X(\sigma)^+$.
  \end{example}

As for (two-sided) recognizability, the definition of
one-sided recognizability was given in a different
form in the articles of Moss\'e \cite{Mosse1992,Mosse1996}.

Let $\sigma\colon A^*\to B^*$ be a  morphism,
let $x\in A^\N$ (\resp $x\in A^\Z$) be such that $y=\sigma(x)$ is infinite
(\resp two-sided infinite). Set
\begin{displaymath}
  C^+(x)=\{|\sigma(x_{[0,n)})|\mid n\ge 0\}.
\end{displaymath}
For $N\ge 1$,
let us say that
$\sigma$ is \emph{one-sided recognizable in the sense of Moss\'e} for $x$
with scope $N$ if
  for $i,j\ge 0$
(\resp $i,j\in\Z$), whenever
$y_{[i,i+N)}=y_{[j,j+N)}$, then $i\in C^+(x)\Leftrightarrow j\in C^+(x)$
    (\resp $i\in C(x)\Leftrightarrow j\in C(x)$).

    Note that if $\sigma$ is one-sided recognizable in the sense of Moss\'e
    for $x\in A^\Z$, then it is one-sided recognizable in the sense
    of Moss\'e for $x^+$.

    The following statement relates the two notions of recognizability
    in the sense of Moss\'e.  A set $U$ of  words is a
    \emph{suffix code} if no element of $U$ is a suffix of another one.
    In particular, the words in $U$ are nonempty.
    \begin{proposition}\label{proposition12Mosse}
      Let $\sigma\colon A^*\to B^*$ be a morphism and let $x\in A^\Z$
      be such that $\sigma(x)$ is two-sided infinite.
      \begin{enumerate}
      \item If $\sigma$ is one-sided recognizable in the sense of Moss\'e
        for $x$, it is recognizable in the sense of Moss\'e for $x$.
      \item If  $\sigma(A)$ is a suffix code and if $\sigma$
        is recognizable in the sense of Moss\'e for $x$,
        then it is one-sided recognizable in the sense of Moss\'e
        for $x$.
        \end{enumerate}
    \end{proposition}
    \begin{proof}
      Set $y=\sigma(x)$.
      
      1. Assume that $\sigma$ is one-sided recognizable
      in the sense of Moss\'e for $x$ with scope $N$. Suppose that $y_{[i-N,i+N]}=y_{[j-N,j+N]}$. Then $y_{[i,i+N)}=y_{[j,j+N)}$
          and thus $i\in C(x)\Leftrightarrow j\in C(x)$.

          2. Assume that $y_{[i,i+M)}=y_{[j,j+M)}$ for some $M\ge 1$
              and $i\in C(x)$. If
              $M$ is chosen large enough, there is an $L\ge 1$
              such that $i+L\in C(x)$ with
              \begin{displaymath}
                i<i+L-N<i+L<i+L+N<i+M.
              \end{displaymath}
              Since $y_{[i+L-N,i+L+N]}=y_{[j+L-N,j+L+N]}$, we have $j+L\in C(x)$.
              But since 
              $\sigma(A)$ is a suffix code, it implies
              $j\in C(x)$ and thus $\sigma$ is one-sided recognizable
              in the sense of Moss\'e.
    \end{proof}
    Note that in condition 2 in Proposition~\ref{propositionRightMarked},
    the condition that $\sigma$ is right-marked
    could not be replaced by the weaker
    condition 2 in Proposition~\ref{proposition12Mosse}
    that $\sigma(A)$ is a suffix code
        (see Example~\ref{exampleContrEx}).

    We now prove the following statement, in part analogous to Proposition
    \ref{propositionTwoDefsRec}. We say that $\sigma\colon A^*\to B^*$
    is \emph{weakly one-sided recognizable} on a one-sided shift $X$
    if
    %$\sigma$ is injective on $X$ and if $\sigma(X)\cap S^k\sigma(X)$
    %implies $S^k\sigma(X)=\sigma(X)$. This condition is equivalent
    %to the fact that
    for every $x\in X$, the sequence $y=\sigma(x)$
    has no other $\sigma$-representation than $(x,0)$.
\begin{proposition}\label{propositionWeakRec}
  Let $\sigma\colon A^*\to B^*$ be a morphism,
   let $x\in A^\N$ and let $X$ be the one-sided shift generated by $x$.
  The following assertions hold.
  \begin{enumerate}
  \item[\rm(i)] If $\sigma$ is one-sided recognizable
    on $X$, it  is one-sided recognizable in the sense of Moss\'e
    for $x$.
  \item[\rm (ii)] If $X$ is minimal, $\sigma$ is injective on $A$
    and $\sigma$ is one-sided recognizable in the sense of Moss\'e
    for $x$, then $\sigma$ is weakly one-sided recognizable on $X$.
    \end{enumerate}
\end{proposition}
\begin{proof}
  (i) Set $y=\sigma(x)$. Since $\sigma$ is one-sided recognizable on $X$,
  its restriction to $X$ is a homeomorphism from $(X^\sigma,T)$ onto
  the closure $Y$ under the shift of $\sigma(X)$.
  Let $\hat{\sigma}$ be the morphism from $(X^\sigma,T)$ onto $Y$
  defined by $\hat{\sigma}(x,i)=S^i\sigma(x)$ where $(X_\sigma,T)$
  is defined by \eqref{eqXsigma}.
  Since $\hat{\sigma}$ is a homeomorphism, there is an integer $N$ such that for every $z,z'\in Y$ with $z=\hat{\sigma}(t,i)$ and $z'=\hat{\sigma}(t',i')$
  such that $z_{[0,N)}=z'_{[0,N)}$, we have $(t_0,i)=
      (t'_0,i')$. In particular,
      $z\in \sigma(X)$ if and only if
      $z'\in\sigma(X)$. Suppose that $i\in C^+(x)$ and $y_{[i,i+N)}=y_{[j,j+N)}$.
        Let $m\ge 0$ be such that $\sigma(x_{[0,m)})=y_{[0,j-k)}$
        with $0\le k<|\sigma(x_m)|$.
        Since $S^i(y)$ is in $\sigma(X)$, we have also $S^j(y)\in \sigma(X)$
        and thus $S^j(y)$ has a $\sigma$-representation $(x',0)$ with $x'\in X$.
        Since $\sigma$ is one-sided recognizable on $X$,
        $S^j(y)$ cannot have a $\sigma$-representation $(S^m(x),k)$
        with  $k>0$. This implies that $k=0$ and thus $j$ is in $C^+(x)$.

        (ii) We follow the same steps as in~\cite{BertheSteinerThuswaldnerYassawi2019}. The proof for the first ones (1 to 4)
        is the same and we don't reproduce it.

        Claim 1. For every $z\in B^\N$, one has $z\in\sigma(X)$
        if and only if for every sequence $m_i$ such that $S^{m_i}(y)$
        converges to $z$, one has $m_i\in C^+(x)$ for all large enough $i$.

        Claim 2. The set $\sigma(X)$ is clopen.

        Claim 3. For every $x'\in X$ and $m\ge 0$, one has $S^m(x')\in\sigma(X)$
        if and only if $m\in C^+(x')$.

        Claim 4. $\sigma$ is a homeomorphism from $X$ onto $\sigma(X)$.

        Assume now that $y'=\sigma(x')$ with $x'\in X$ and that $y'=S^k\sigma(x'')$ with $x''\in X$ and $0\le k<|\sigma(x''_0)|$. By Claim 3,
        we have $k\in C^+(x'')$, which forces $k=0$. Finally, by Claim 4,
        we obtain $x'=x''$.
\end{proof}
Observe that  Assertion (ii) is weaker than the corresponding assertion
in Proposition~\ref{propositionTwoDefsRec} since the conclusion
is not that $\sigma$ is one-sided recognizable on $X$. This
is actually not true, as shown in the following example.
\begin{example}\label{exampleContrEx}
  Let $\sigma\colon a\mapsto ba,b\mapsto aa$
  (note that $X(\sigma)$ is the period-doubling shift of Example
  \ref{examplePeriodDoubling}). Since $\sigma$ is
  primitive and aperiodic, it is recognizable on $X(\sigma)$.
  Thus, by Proposition~\ref{propositionTwoDefsRec}, it is recognizable in the sense
  of Moss\'e for every $x\in X$. Since $\sigma(A)$ is a suffix code,
  this implies, by Proposition~\ref{proposition12Mosse},
  that $\sigma$ is one-sided recognizable in the
  sense of Moss\'e for every $x\in X$. This implies
  in turn by Proposition~\ref{propositionWeakRec} that $\sigma$ is weakly
  one-sided recognizable on $X(\sigma)^+$.
  
  It is however not
  one-sided recognizable on $X(\sigma)^+$ because the sequence
  $y=a\sigma(a)\sigma^2(a)\cdots$ is such that $y=a\sigma(y)$
  and consequently has the two $\sigma$-representations
  $(ay,1)$ and $(by,1)$.
  \end{example}

There are primitive morphisms which are not one-sided
recognizable on $X^+$ (see Examples \ref{exampleContrEx} and~\ref{exampleAntiFibo}).
The following result characterizes morphisms which
are not one-sided recognizable. It is closely
related with the main result of \cite{AkiyamaTanYuasa2017} (see the comment
after the proof).
\begin{theorem}\label{theoremCharactOneSided}
  Let $\sigma\colon A^*\to B^*$ be a morphism and let
  $X$ be a shift space on $A$  with at most a finite number
  of periodic points and such that $\sigma$ is recognizable
  on $X$ at aperiodic points. Let $Y$ be the closure of
  $\sigma(X)$ under the shift and let $y\in Y^+$
  be an aperiodic point.
  
  The morphism  $\sigma$ is not one-sided
  recognizable on $X^+$ at  $y$ if and only if
  there are words $u,u'\in A^+$ and $v\in B^+$, and
  a one-sided sequence  $x\in X^+$ such that (see Figure~\ref{figureNotOneSided})
  \begin{enumerate}
    \item[\rm(i)] $ux,u'x\in X^+$,
    \item[\rm(ii)] $y=v\sigma(x)$ with $v$ a suffix of $\sigma(u)$
      and $\sigma(u')$,
      \item[\rm(iii)] the last letters of $u,u'$ are distinct.
      \end{enumerate}
  In particular, $y$ is a shift of a left-special sequence
  in $Y^+$.
\end{theorem}
\begin{proof}
  \begin{figure}[hbt]
    \centering
    \tikzset{node/.style={circle,draw,minimum size=0.1cm,inner sep=0pt}}
    \tikzset{title/.style={minimum size=0.1cm,inner sep=0pt}}
    \begin{tikzpicture}
      \node[node](u)at(-1.5,.5){};
      \node[node](w)at(-1.5,-.5){};
      \node[node](y)at(0,0){};
      \node[node](z)at(1,0){};
      \node[title]at(5,0){$y$};

      \draw[above](y)edge node{$v$}(z);
      \draw[bend left,above](u)edge node{$u$}(z);
      \draw[above](u)edge node{}(y);
      \draw[bend right,below](w)edge node{$u'$}(z);
      \draw[below](w)edge node{}(y);
      \draw[bend left,above](z)edge node{}(2,.7);
      \draw[above](2,.7)edge node{$x$}(4,.7);
      \draw[above](z)edge node{$\sigma(x)$}(4,0);
    \end{tikzpicture}
    \caption{The morphism $\sigma$ is not one-sided recognizable at $y$.}
    \label{figureNotOneSided}
    \end{figure}
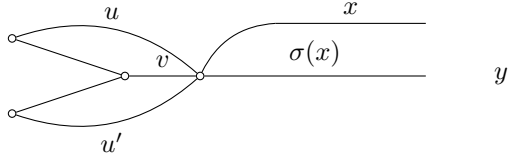
  Let $y\in B^\N$ be an aperiodic point
  with two distinct $\sigma$-representations
$(z,k)$ and $(z',k')$ with $z,z'\in X^+$. 
Note that, since $y$ is aperiodic,
  we cannot have $z=z'$. 

  Let $t,t'\in X$ be such
  that $t^+=z$ and $t'^+=z'$. Since $X$ contains a finite number
  of periodic points, their complement is a shift-invariant
  open set. Thus there is a clopen set $U\subset X$
  containing the orbits of $t,t'$ such that $\sigma$ is recognizable on $U$.
  Let $V$ be the closure of $\sigma(U)$ under the shift. Then $\hat{\sigma}$
  is a homeomorphism from $U^\sigma=\{(u,k)\mid u\in U, 0\le k<|\sigma(u_0)|\}$
  onto $V$. Let $N$ be the window size of the block map defining
  the restriction of $\hat{\sigma}$ to $U^\sigma$.

   Since $\sigma(t)_{k+i}=\sigma(t')_{k'+i}=y_i$
  for every $i\ge 0$, 
  we have $t_{k+j}=t'_{k'+j}$ for all $j\ge N$. Since $z\ne z'$,
  there is an index $n$ with $0\le n<N$ such that $z_{k+n}\ne z'_{k'+n}$.
  We choose $n$ maximal.
  Set
  \begin{eqnarray*}
    x&=&S^{k+n}z=S^{k'+n}z',\quad u=z_{[0,k+n)},\quad u'=z'_{[0,k'+n)},\\
        v&=&S^k\sigma(u)=S^{k'}\sigma(u').
    \end{eqnarray*}
  It is then easy to verify that conditions (i), (ii) and (iii)
  are satisfied. Moreover, since $t,t'$ are right asympotic,
  $y$ is a shift of a left-special sequence in $Y^+$.

  Conversely, set $\sigma(u)=pv$ and $\sigma(u')=p'v$. We may assume
  that $p,p'$ are proper prefixes of the image by $\sigma$
  of the first letters of $u$ and $u'$ respectively
  (otherwise we can shorten $u$ or $u'$ by one letter).
  Let $k=|\sigma(p)|$ and $k'=|\sigma(p')|$. Set $z=ux$ and $z'=u'x$.
  Then $z,z'\in X^+$, $0\le k<|\sigma(z_0)|$ and $0\le k'<|\sigma(z'_0)$.
  Then
  \begin{displaymath}
    y=S^k\sigma(z)=S^{k'}\sigma(z')
  \end{displaymath}
  and thus $(z,k)$ and $(z',k')$ are two distinct $\sigma$-representations of $y$.
\end{proof}

Note first that, by Theorem~\ref{theoremRecognizabilityAperiodic},
the hypotheses of Theorem~\ref{theoremCharactOneSided}
are satisfied when $\sigma\colon A^*\to A^*$
is an endomorphism and $X=X(\sigma)$.

Note also the connection with the main result of \cite{AkiyamaTanYuasa2017}.
By \cite[Theorem 1.1]{AkiyamaTanYuasa2017}, a primitive morphism
$\sigma\colon A^*\to A^*$ with a fixed point $x$
is not one-sided recognizable in the sense of Moss\'e if
and only if for every $N\ge 0$ there are $i,j\ge 0$ such
that $\sigma(x_{[i+1,i+N]})=\sigma(x_{[j+1,j+N]})$
with $\sigma(x_i)$ a proper suffix of $\sigma(x_j)$.

Note that the number of shifts
of a left-special sequence at which a morphism  fails to
be  one-sided recognizable can be arbitrary large
(see Example~\ref{exampleAntiFibo}).

We say that a morphism $\sigma\colon A^*\to B^*$ is
\emph{almost one-sided recognizable}
on a one-sided shift $X$ if $\sigma$ is  one-sided recognizable on $X$ except
at a finite number of points of  $B^\N$.

In the case of an endomorphism, we have the following more precise
statement.
\begin{theorem}\label{theoremAlmostRec}
  Let $\sigma\colon A^*\to A^*$ be a  morphism such that
  $X(\sigma)$ is minimal aperiodic.
  Then $\sigma$ is almost one-sided recognizable on $X(\sigma)^+$.
\end{theorem}
\begin{proof}
  Since $\sigma$ is minimal and aperiodic, it is recognizable
  on $X(\sigma)$ by Theorem~\ref{theoremRecognizabilityAperiodic}.
  Then the map $x\in X(\sigma)\mapsto \sigma(x)$
  is injective and thus the restriction of $\sigma$ to
  $X(\sigma)$ is a homeomorphism from $X(\sigma)$ onto
  $\sigma(X(\sigma))$. This implies that there is an integer $N\ge 1$
  such that for every $x,x'\in X(\sigma)$,
  if $\sigma(x)_{[-N,N]}=\sigma(x')_{[-N,N]}$ then $x_0=x'_0$.

  We will show that if $y\in X(\sigma)^+$ has several $\sigma$-representations,
  then $y=S^i(t)$ with $t\in X(\sigma)^+$ left-special and $i\le N$.
  By the remark following Proposition~\ref{propositionLRisLinear}, this implies our conclusion.

  Let $y\in X(\sigma)^+$ have two distinct $\sigma$-representations
  $(x,k)$ and $(x',k')$. Note that, since $X(\sigma)$ is aperiodic,
  we cannot have $x=x'$. Then $y=S^k(\sigma(x))=S^{k'}(\sigma(x'))$
  with $x,x'\in X(\sigma)^+$. Let $z,z'\in X(\sigma)$ be such
  that $z^+=x$ and $z'^+=x'$.

  If $S^k\sigma(z)_{[-N,N]}=S^{k'}\sigma(z')_{[-N,N]}$,
  then $S^k\sigma(z)_{[-N+i,N+i]}=S^{k'}\sigma(z')_{[-N+i,N+i]}$ for all $i\ge 0$
  (because $S^k(\sigma(z))_j=S^{k'}(\sigma(z'))_j=y_j$ for all $j\ge 0$)
  and thus $x=x'$, a contradiction. This implies that
  for some $i\le N$, we have $S^k\sigma(z)_{-i}\ne S^{k'}\sigma(z')_{-i}$.
  We choose $i$ minimal. 
  Then  $t=S^{k-i}\sigma(z)^+=S^{k'-i}\sigma(z')^+$ is left special
  and such that $S^i(t)=y$.
  
\end{proof}

\begin{example}\label{exampleAntiFibo}
  Let $\sigma\colon a\mapsto ab,b\mapsto a$ be the Fibonacci morphism.
  The morphism $\tilde{\sigma}\colon a\mapsto ba,b\mapsto a$ is not one-sided
  recognizable. Since $a\tilde{\sigma}(x)=\sigma(x)a$
  for every $x\in A^*$, the shift $X(\tilde{\sigma})$ is equal to the Fibonacci shift.
  Since $X(\sigma)$ is a Sturmian shift, there is a unique
  left-special sequence, which is $t=\sigma^\omega(a)$.
  Since $\sigma(t)=t$, we have also
  $a\tilde{\sigma}(t)=t$.

  Accordingly, the morphism
  $\tilde{\sigma}$ is not one-sided recognizable at $t$, since
  \begin{displaymath}
    t=\tilde{\sigma}(bt)=S\tilde{\sigma}(at).
  \end{displaymath}
  The first equality comes from
  \begin{displaymath}
    \tilde{\sigma}(bt)=a\tilde{\sigma}(t)=t,
  \end{displaymath}
  and the second one from
  \begin{displaymath}
\tilde{\sigma}(at)=ba\tilde{\sigma}(t)=bt.
  \end{displaymath}
  Consider now the morphism $\tilde{\sigma}^n$. We have
  \begin{eqnarray*}
    \tilde{\sigma}^n(at)=&\tilde{\sigma}^{n-1}(b)&\tilde{\sigma}^{n-1}(a)\tilde{\sigma}^n(t)=\tilde{\sigma}^{n-1}(b)\tilde{\sigma}^{n-1}(t)\\
    \tilde{\sigma}^n(bt)=&       &\tilde{\sigma}^{n-1}(a)\tilde{\sigma}^n(t).
  \end{eqnarray*}
  Set $F_n=|\tilde{\sigma}^{n-1}(b)|$.
  Then, for $0\le i< F_{n+1}$, the sequence $S^i\tilde{\sigma}^n(bt)$
  has the two $\tilde{\sigma}^n$-representations
  \begin{displaymath}
    (\tilde{\sigma}^n(bt),i)\mbox{ and } (\tilde{\sigma}^{n}(at),F_n+i).
    \end{displaymath}
  This shows that the number of shifts of a left-special
  sequence at which a morphism 
  fails to be one-sided recognizable can be arbitrary large.
\end{example}

The following example (due to Fabien Durand \cite{Durand2021})
shows that, for a non minimal morphism,
the set of left-special right-infinite sequences may be infinite
and that Theorem~\ref{theoremAlmostRec} may be false.

  \begin{figure}[hbt]
    \centering
    \tikzset{node/.style={circle,draw,minimum size=0.1cm,inner sep=0pt}}
    \begin{tikzpicture}
      \node[node](0)at(0,0){};
      \node[node](a)at(1,1){};\node[node](b)at(1,0){};\node[node](c)at(1,-1){};
      \node[node](ab)at(2,2){};\node[node](ac)at(2,1){};
      \node[node](ba)at(2,0){};
      \node[node](ca)at(2,-1){};\node[node](cc)at(2,-2){};
      \node[node](aba)at(3,3){};
      \node[node](aca)at(3,2){};\node[node](acc)at(3,1){};
      \node[node](bab)at(3,0.5){};\node[node](bac)at(3,-0.5){};
      \node[node](cab)at(3,-1){};
      \node[node](cca)at(3,-2){};\node[node](ccc)at(3,-3){};
      \draw[above](0)edge node{$a$}(a);
      \draw[above](a)edge node{$b$}(ab);
      \draw[above](0)edge node{$b$}(b);
      \draw[above](0)edge node{$c$}(c);
      \draw[above](a)edge node{$c$}(ac);
      \draw[above](b)edge node{$a$}(ba);
      \draw[above](c)edge node{$a$}(ca);
      \draw[above](c)edge node{$c$}(cc);
      \draw[above](ab)edge node{$a$}(aba);
      \draw[above](ac)edge node{$a$}(aca);
      \draw[above](ac)edge node{$c$}(acc);
      \draw[above](ba)edge node{$b$}(bab);\draw[above](ba)edge node{$c$}(bac);
      \draw[above](ca)edge node{$b$}(cab);
      \draw[above](cc)edge node{$a$}(cca);\draw[above](cc)edge node{$c$}(ccc);

      %left-special
      \node[node](0)at(6,1){};
      \node[node](a)at(7,2){};\node[node](c)at(7,-1){};
      \node[node](ab)at(8,2){};
      \node[node](ca)at(8,0){};\node[node](cc)at(8,-2){};
      \node[node](aba)at(9,2){};
      \node[node](cab)at(9,0){};
      \node[node](cca)at(9,-1){};\node[node](ccc)at(9,-3){};
      \node[node](abac)at(10,2){};
      \node[node,fill=red](caba)at(10,0){};
      \node[node](ccab)at(10,-1){};
      \node[node](ccca)at(10,-2){};
      \node[node](cccc)at(10,-3){};
      \node[node](abaca)at(11,2){};
      \node[node](ccaba)at(11,-1){};

      \draw[above](0)edge node{$a$}(a);\draw[above](0)edge node{$c$}(c);
      \draw[above](a)edge node{$b$}(ab);
      \draw[above](c)edge node{$a$}(ca);\draw[above](c)edge node{$c$}(cc);
      \draw[above](ab)edge node{$a$}(aba);
      \draw[above](ca)edge node{$b$}(cab);
      \draw[above](cc)edge node{$a$}(cca);\draw[above](cc)edge node{$c$}(ccc);
      \draw[above](aba)edge node{$c$}(abac);
      \draw[above](cab)edge node{$a$}(caba);
      \draw[above](cca)edge node{$b$}(ccab);
      \draw[above](ccc)edge node{$a$}(ccca);
      \draw[above](ccc)edge node{$c$}(cccc);
      \draw[above](abac)edge node{$a$}(abaca);
      \draw[above](ccab)edge node{$a$}(ccaba);
      
    \end{tikzpicture}
    \caption{The set $\cL(X(\sigma))$ and the tree of left-special words.}
    \label{figureLeftSpec}
  \end{figure}
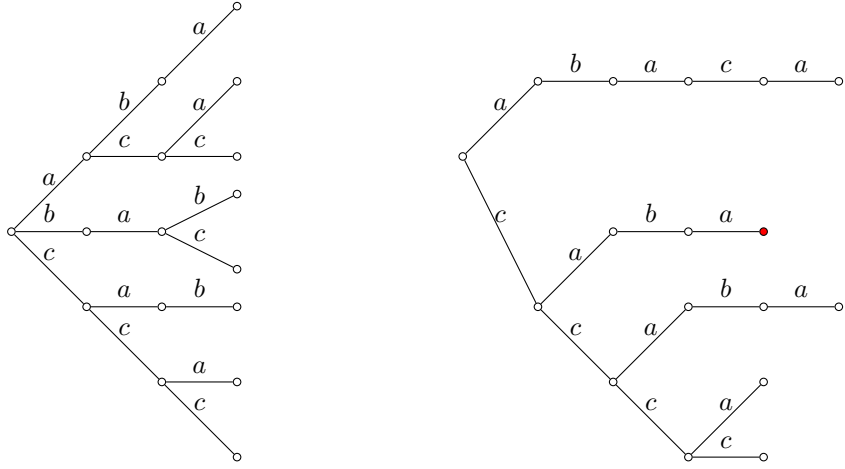
  \begin{example}
  Let $\sigma\colon a\mapsto abac,b\mapsto ab,c\mapsto c$.
  The left-special right-infinite sequences are the
  $\sigma^\omega(c^na)$ for $n\ne 1$ (see Figure~\ref{figureLeftSpec}).

  Indeed, $a$ is
  left-special since $ba,ca\in\cL(X(\sigma))$ and thus
  $\sigma^\omega(a)$ is left-special. Next $c^na$ is left-special
  since $c^{n+1}a,ac^na\in\cL(X(\sigma))$ and thus $\sigma^\omega(c^{n+1}a)$
  is left-special.

  The morphism $\sigma$ is not one-sided recognizable on $X(\sigma)^+$
  at each point $c^n\sigma^\omega(a)$ for $n\ge 1$ since.
  \begin{displaymath}
c^n\sigma^\omega(a)=\sigma(c^n\sigma^\omega(a))=S^3(\sigma(ac^{n-1}\sigma^\omega(a))).
    \end{displaymath}
\end{example}
The following picture summarizes the relations between the various notions
of recognizability.

  \begin{figure}[hbt]
    \centering
    \tikzset{node/.style={draw,minimum size=0.1cm,inner sep=8pt}}
    \tikzset{title/.style={minimum size=0.1cm,inner sep=0pt}}
    \begin{tikzpicture}
      \node[node](2rec)at(0,2.3){recognizable};
      \node[node](2recM)at(6,2.3){ Moss\'e recognizable};

      \draw[->,above,double](2rec.south east)-- node{Prop. \ref{propositionTwoDefsRec}}(2recM.south west);
      \draw[->,above,double](2recM.north west)--node{$X$ minimal, $\sigma$ injective on $A$}(2rec.north east);
      %\node[title]at(3,3.2){$X$ minimal,};
      %\node[title]at(3,2.8){$\sigma$ injective on $A$};
      %\node[title](implies)at(3,2.5){$\Longleftarrow$};
      %\node[title]at(3,2.2){Prop. \ref{propositionTwoDefsRec}};
      %\node[title]at(3,1.8){$\Longrightarrow$};

      \node[node](1rec)at(0,0){one-sided recognizable};
      \node[node](w1rec)at(0,-1.5){weak one-sided recognizable};
      %\node[title]at(1,1){$\Uparrow$};
      %\node[title](dimplies)at(-2.1,1){$\sigma$ right-marked $\Downarrow$};
      %\node[title]at(0,1){Prop. \ref{propositionRightMarked}};
      \node[node](1recM)at(6.5,-.5){one-sided Moss\'e recognizable};
      \draw[->,left,double](2rec.south west)-- node{$\sigma$ right-marked}(1rec.north west);
      \draw[->,left,double](1rec.north east)--node{Prop. \ref{propositionRightMarked}}(2rec.south east);
      %\node[title]at(3,0){$\Longrightarrow$};
      %\node[title]at(3,-.5){Prop. \ref{propositionWeakRec}};
      \draw[->,below,double](1rec)-- node{Prop. \ref{propositionWeakRec}}(1recM.north west);
      %\node[title]at(3,-1.1){$\Longleftarrow$};
      %\node[title]at(3,-1.4){$\sigma$ injective on $A$};
      \draw[->,below,double](1recM.south west)-- node{$\sigma$ injective on $A$}(w1rec);
      \draw[->,left,double](2recM.south west)--node{$\sigma(A)$ suffix}(1recM.north west);
      %\node[title]at(4.5,1){$\sigma(A)$ suffix};
      %\node[title]at(5.5,1){$\Downarrow$};
      \draw[->,left,double](1recM.north east)--node{Prop. \ref{proposition12Mosse}}(2recM.south east);
      %\node[title]at(6.5,1){Prop. \ref{proposition12Mosse}};
      %\node[title]at(7.5,1){$\Uparrow$};
      \draw[->,double](1rec)--node{}(w1rec);
      %\node[title]at(0,-.5){$\Downarrow$};
      
    \end{tikzpicture}
    \caption{The various notions of recognizability.}\label{figureVarious}
  \end{figure}
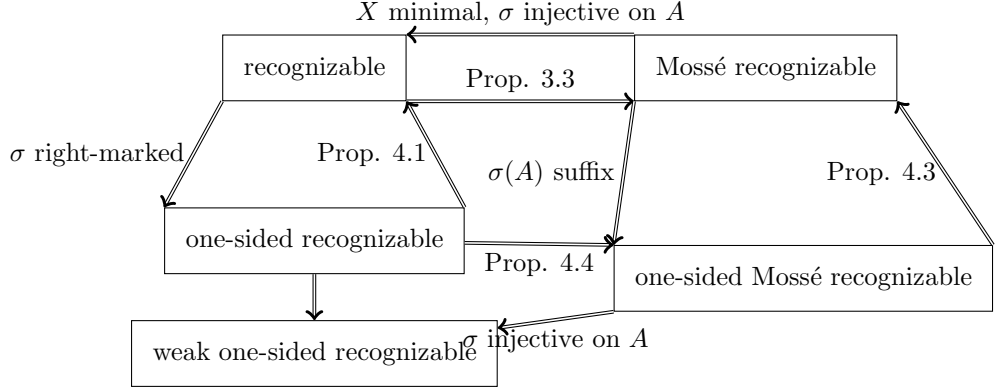
  %%%%%%%%%%%%%%
  \section{Continuous eigenvalues}
  Let $(X,T)$ be a topological dynamical system. A complex number $\lambda$
  is a \emph{continuous eigenvalue} of $(X,T)$ if there exists a continuous
  function $f\colon X\to{\mathbb C}$ with $f\ne 0$ such that $f\circ T=\lambda f$.
  Such a function is called a \emph{continuous eigenfunction}.
  The \emph{continuous spectrum} of $(X,T)$ is the set
  of continuous eignevalues of $(X,T)$. It is invariant
  under conjugacy and contains always the value $\lambda=1$
  since a constant function is continuous.
\begin{proposition}
  If $(X,T)$ is recurrent, every continuous eigenvalue is of modulus $1$
  and every continuous eigenfunction has constant modulus.
\end{proposition}
\begin{proof}
  Let $f$ be a continuous eigenfunction corresponding to $\lambda$.
  Since $(X,T)$ is recurrent, there is a point $x$
  such that the set of $T^n(x)$ for $n\ge 0$
  is dense in $X$. Thus for every $\varepsilon>0$ there
  is an infinity of $n$ such that $d(T^n(x),x)\le\varepsilon$. Since $f$
  is continuous, this forces $|\lambda|=1$. Next,
  since $|f(T(x))|=|\lambda||f(x)|=|f(x)|$ we obtain that
  $|f(x)|$ is constant since $f$ is continuous.
\end{proof}
\begin{example}
  Let $\sigma\colon a\mapsto ab, b\mapsto ba$ be the Thue-Morse
  morphism. The map $f\colon X(\sigma)\to \{-1,1\}$ defined by
  $f(x)=(-1)^i$ if $x$ has a $\sigma$-representation $(y,i)$
  with $y\in X(\sigma)$ is a continuous map such that
  $f(S(x))=-f(x)$. Thus $-1$ is a continuous eigenvalue of $(X(\sigma),S)$.
\end{example}
The following statement also appears as \cite[Proposition 2.1]{BertheCecchiYassawi2022}.
\begin{proposition}\label{propositionSpectra12}
  For every recurrent shift space $X$, 
the continuous spectrum of $X$ and $X^+$ are equal.
\end{proposition}
\begin{proof}
  Suppose first that $f$ is a continuous eigenfunction of $X^+$
  for the eigenvalue $\lambda$. Then the map $g\colon X\to \CC$
  defined by $g(x)=f(x^+)$ is a continuous eigenfunction of $X$
  for the same eigenvalue.

  Conversely, let $g$ be a continuous eigenfunction of $X$ for the
  eigenvalue $\lambda$. If $x,x'\in X$ are right-asymptotic, then
  $g(x)=g(x')$. Indeed, since $g$ is continuous, for every $\varepsilon>0$
  there is an $N\ge 1$ such that $y_{[-N,N]}=y'_{[-N,N]}$
  implies $|g(y)-g(y')|<\varepsilon$.
  If $S^n(x)^+=S^n(x')^+$, we have 
  $S^m(x)_{[-N,N]}=S^m(x')_{[-N,N]}$ for $m\ge n+N$. 
      This implies that $|g(S^m(x))-g(S^m(x'))|<\varepsilon$,
      whence $g(S^m(x))=g(S^m(x'))$ and thus that
      $g(x)=g(x')$. We can then define $f\colon X^+\to \CC$
      by $f(y)=g(x)$ for some $x$ such that $x^+=y$.
      The map $f$ is clearly an eigenfunction of $X^+$
      for the eigenvalue $\lambda$.
\end{proof}

\begin{example}
  Consider again the morphism $\sigma\colon a\mapsto ba,b\mapsto aa$
  of Example~\ref{exampleContrEx}. As for the Thue-Morse
  morphism, $-1$ is an eigenvalue of $X(\sigma)$
  and a corresponding eigenfunction 
  is $f(x)=(-1)^i$ if $x$ has a $\sigma$-representation
  $(y,i)$ with $y\in X(\sigma)$ and $0\le i<2$.
  Note that, since $\sigma$ is weakly one-sided recognizable,
  the restriction of $f$ to $X(\sigma)^+$ is also
  an eigenfunction.
  \end{example}

\bibliographystyle{plain}
\bibliography{oneSided}
\end{document}